\documentclass{amsart}

\usepackage{amssymb,mathrsfs}

\makeatletter
\@namedef{subjclassname@2010}{%
  \textup{2010} Mathematics Subject Classification}
\makeatother

\newtheorem{thm}{Theorem}[section]

\newtheorem{lem}[thm]{Lemma}
\newtheorem{prop}[thm]{Proposition}

\theoremstyle{definition}

\numberwithin{equation}{section}

\begin{document}

\title[Value-distributions of Dedekind zeta functions]{On the value-distributions of logarithmic derivatives of Dedekind zeta functions}

\author[M. Mine]{Masahiro Mine}
\address{Department of Mathematics\\ Tokyo Institute of Technology\\ 
2-12-1 $\hat{\text{O}}$okayama, Meguro-ku, Tokyo 152-8551, Japan}
\email{mine.m.aa@m.titech.ac.jp}

\date{}

\begin{abstract}
We study the distributions of values of the logarithmic derivatives of the Dedekind zeta  functions on a fixed vertical line. 
The main object is determining and investigating the density functions of such value-distributions for any algebraic number field. 
We construct the density functions as the Fourier inverse transformations of certain functions represented by infinite products that come from the Euler products of the Dedekind zeta functions. 
\end{abstract}



\maketitle

\section{Introduction}\label{sec1}
Let $s=\sigma+it$ be a complex variable, and let $\zeta(s)$ denote the Riemann zeta function. 
The study of the value-distribution of $\zeta(s)$ is a classical topic in number theory.
After several results on the values of $\log\zeta(s)$ and $(\zeta'/\zeta)(s)$, Bohr and Jessen \cite{BohrJessen} arrived at the following result. 
Fix a rectangle $R$ in the complex plane whose edges are parallel to the axis. 
For any $T>0$, let $V_\sigma(T,R)$ be the Lebesgue measure of the set of all $t\in[-T,T]$ for which $\log\zeta(\sigma+it)$ belongs to $R$. 
Then they proved the existence of the limit 
\begin{equation}\label{eq1.1}
W_\sigma(R)=\lim_{T\to\infty}\frac{1}{2T}V_\sigma(T,R)
\end{equation}
for any $\sigma>1/2$. 
The method of Bohr and Jessen was improved by Jessen and Wintner \cite{JessenWintner}. 
Kershner and Wintner \cite{KershnerWintner} proved an analogous result for $(\zeta'/\zeta)(s)$. 

Bohr and Jessen also proved that the limit $W_\sigma(R)$ can be written as an integral containing a certain continuous and non-negative function $\mathcal{F}_\sigma(z)$ as follows: 
\begin{equation}\label{eq1.2}
W_\sigma(R)=\int_R\mathcal{F}_\sigma(z)dxdy,
\end{equation}
where $z=x+iy$.
The existence of the limit \eqref{eq1.1} was generalized to $L$-functions of cusp forms and Dedekind zeta functions by Matsumoto \cite{Matsumoto1, Matsumoto2, Matsumoto3}. 
On the other hand, the generalization of the integral formula \eqref{eq1.2} has been restricted within the case of Dedekind zeta functions associated with Galois extensions of $\mathbb{Q}$ \cite{Matsumoto3}. 

Then, in the last decade Ihara studied the various averages of the values of $\log L(s,\chi)$ and $(L'/L)(s,\chi)$ of $L$-functions $L(s,\chi)$ associated with characters $\chi$.  
Let $K$ be any global field and $\chi$ a character on $K$, and fix $s=\sigma+i\tau$. 
In \cite{Ihara}, he considered the existence of the function $M_\sigma(z)$ for which the formula 
\begin{equation}\label{eq1.3}
\textrm{Ave}_\chi\Phi\Big(\frac{L'(s,\chi)}{L(s,\chi)}\Big)
=\int_\mathbb{C}\Phi(z)M_\sigma(z)|dz|
\end{equation}
holds for a sufficiently wide class of test functions $\Phi$, where $|dz|$ denotes the measure $(2\pi)^{-1}dxdy$. 
The meanings of the averages $\textrm{Ave}_\chi$ are given in \cite{Ihara} for the three cases: 

(A) The field $K$ is either $\mathbb{Q}$, an imaginary quadratic field, or a function 
field over $\mathbb{F}_q$. 
The character $\chi$ runs over Dirichlet characters on $K$; 

(B) The field $K$ is a number field having at least two archimedean primes, and
$\chi$ runs over normalized unramified Gr\"{o}ssencharacters, 
which are defined in \cite{Ihara}, of $K$;

(C) $K=\mathbb{Q}$ and $\chi=\chi_t$, where $t\in\mathbb{R}$ 
and $\chi_t(p)=p^{-it}$ for each prime number $p$. 

\noindent
Here we refer only to the case (C), and there 
\[\textrm{Ave}_\chi\Phi\Big(\frac{L'(s,\chi)}{L(s,\chi)}\Big) 
=\lim_{T\to\infty}\frac{1}{2T}\int_{-T}^T\Phi\Big(\frac{\zeta'}{\zeta}(s+it)\Big)dt.\]
If we take $\Phi$ as the characteristic function of a rectangle $R$, we obtain a result analogous to the integral formula \eqref{eq1.2} for $(\zeta'/\zeta)(s)$. 
In each cases, Ihara constructed $M_\sigma(z)$ and succeeded to prove the formula \eqref{eq1.3} for $\sigma>1$ or $\sigma>3/4$ if $K$ is a number field or a functional field, respectively.

Ihara also studied the Fourier transform $\tilde{M}_\sigma(w)$ defined by 
\[\tilde{M}_\sigma(w)=\int_\mathbb{C}M_\sigma(z)\psi_w(z)|dz|,\]
where 
\begin{equation}\label{eq1.4}
\psi_w(z)=\exp(i\langle z,w\rangle)
\end{equation}
is the additive character of $\mathbb{C}$ defined by using the inner product 
\[\langle z,w\rangle=\mathfrak{R}z\mathfrak{R}w+\mathfrak{I}z\mathfrak{I}w
=\mathfrak{R}(z\overline{w}).\] 
In the above three cases, Ihara noted that the function $\tilde{M}_\sigma(w)$ has an infinite product representation that comes from the Euler product of $L(s,\chi)$. 

The analogue of Ihara's results for $\log L(s,\chi)$ was proved by Ihara and Matsumoto \cite{IharaMatsumoto1, IharaMatsumoto2} in the case (A) and (C). 
Moreover, the condition for $\sigma$ was weakened to $\sigma>1/2$ unconditionally under  some revisions of the meanings of averages and restrictions for a class of test functions \cite{IharaMatsumoto1, IharaMatsumoto2, IharaMatsumoto4}. 
Today there are several variations of the results of Ihara--Matsumoto studied by Ihara--Matsumoto \cite{IharaMatsumoto3}, Lebacque--Zykin \cite{LebacqueZykin}, Matsumoto--Umegaki \cite{MatsumotoUmegaki1}, and Mourtada--Murty \cite{MourtadaMurty}, however, a generalization of the case (C) to an arbitrary number field is not known by a technical difficulty as well as the formula \eqref{eq1.2}. 
We solve such difficulties by the method of Guo \cite{Guo1} below. 

The functions $\mathcal{F}_\sigma$ from Bohr--Jessen and $M_\sigma$ from Ihara--Matsumoto in the case (C) are regarded as the density functions corresponding to the value-distribution of $\log\zeta(\sigma+it)$ or $(\zeta'/\zeta)(\sigma+it)$.
Such a density function had been also studied by Guo \cite{Guo1}, who proved the following. 
Write $\Phi(x+iy)=\phi(x,y)$, where $\phi:\mathbb{R}^2\to\mathbb{R}$ is an infinitely differentiable and compactly supported function. 
Let $\sigma_1$ be a large fixed positive real number, and let $\theta,\delta>0$ with $\delta+3\theta<1/2$. 
Then the following formula holds: 
\[\frac{1}{T}\int_0^T\Phi\Big(-\frac{\zeta'}{\zeta}(\sigma+it)\Big)dt
=\iint_{\mathbb{R}^2}\phi(u,v)\hat{\chi}(u,v)dudv+E\]
for any $\sigma\in[1/2+(\log T)^{-\theta},\sigma_1]$. 
Here $\chi(u,v)$ is the function 
\[\prod_p\int_0^1
\exp\Big(2\pi iu\log p\sum_{m=1}^\infty\frac{\cos(2\pi mt)}{p^{m\sigma}}
-2\pi iv\log p\sum_{m=1}^\infty\frac{\sin(2\pi mt)}{p^{m\sigma}}\Big)dt\]
and $\hat{\chi}$ denotes the Fourier transform of $\chi$. 
The error term $E$ is estimated as 
\[E\ll\exp\big(-\tfrac{1}{4}(\log T)^{\frac{2}{3}\theta}\big)
\int_A|\hat{\phi}(u,v)|dudv
+\int_{\mathbb{R}^2\backslash A}|\hat{\phi}(u,v)|dudv,\]
where $A$ is a square 
\[A=[-(\log T)^\delta,(\log T)^\delta]\times[-(\log T)^\delta,(\log T)^\delta].\]
Guo \cite{Guo2} also proved several results on the density of zeros of $\zeta'(s)$ as applications of his studies on the density function $\hat{\chi}(u,v)$. 

Guo's method is different from the ones of Bohr--Jessen and Ihara--Matsumoto in some aspects (see \S\S\ref{subsec2.3} below). 
By the differences, we may apply Guo's method to cases where the generalization has technical difficulty through the other two methods. 
In the present paper, we prove an analogue of Guo's result in the case of Dedekind zeta function associated with an arbitrary number field, mixed some notations from the papers of Ihara and Matsumoto. 
It seems difficult to prove such a result by the method of \cite{Matsumoto1, Matsumoto2, Matsumoto3}. 
Although our proof is largely same as Guo's original proof, sometimes we arrange the argument for a cost of the treatment of number fields. 

\vspace{3mm}
The paper consists of six sections. 
In Section \ref{sec1} as above, we discussed the previous works to the theory of the value-distributions that we consider. 
In Section \ref{sec2}, we states the results of the paper after setting up some notations. 
The main result is Theorem \ref{thm2.3} bellow, which is an analogue of the result in \cite{Guo1} for the Dedekind zeta functions.
This theorem is deduced from Propositions \ref{prop2.1} and \ref{prop2.2}. 
We also discuss differences between the methods of the present paper and of Bohr--Jessen or Ihara--Matsumoto. 
In Section \ref{sec3}, we prove Proposition \ref{prop2.1}. 
The method of the proof is almost the same as \cite{Guo1}, except for using some results on the density of the zeros of the Dedekind zeta functions. 
In Section \ref{sec4}, we prove Proposition \ref{prop2.2}. 
We prove this proposition by following the method of \cite{Guo1} again. 
At the final step in the proof, it is necessary to give a lower bound on a certain Dirichlet series supported on prime powers, which is more complicated in the case we consider than in \cite{Guo1}. 
To give an available lower bound, we use the Chebotarev density theorem, and this idea is  one of the most important parts at the present paper. 
Section \ref{sec5} is the reminder work to the proof of Theorem \ref{thm2.3}. 
This is just a simple Fourier analysis. 
In Section \ref{sec6}, we discuss further results obtained by the same method as  in the proof of Theorem \ref{thm2.3}. 
We have the analogue of Theorem \ref{thm2.3} for $L$-functions of a primitive Dirichlet character or a primitive cusp form.

\section{Statements of results}\label{sec2}
\subsection{Dedekind zeta functions}
Let $K$ be an algebraic number field of degree $d$.
The Dedekind zeta function of $K$ is defined by
\begin{equation}\label{eq2.1} 
\zeta_K(s)=\sum_\mathfrak{a}\frac{1}{N(\mathfrak{a})^s}
=\prod_\mathfrak{p}(1-N(\mathfrak{p})^{-s})^{-1},
\end{equation}
where $\mathfrak{a}$ and $\mathfrak{p}$ run through all non-zero ideals and prime ideals of $K$, respectively, and $N(\mathfrak{a})$ denotes the norm of $\mathfrak{a}$. 
The series and product converge absolutely for $\sigma>1$, and the function has a holomorphic continuation to the whole complex plane except for a simple pole at $s=1$.
Moreover, it satisfies the functional equation
\[\tilde{\zeta}_K(s)=\tilde{\zeta}_K(1-s),\]
where 
\[\tilde{\zeta}_K(s)
=|D_K|^{\frac{s}{2}}(\pi^{-\frac{s}{2}}\Gamma(\tfrac{s}{2}))^{r_1}
(2(2\pi)^{-s}\Gamma(s))^{r_2}\zeta_K(s).\]
Here $D_K$ denotes the discriminant of $K$, $r_1$ is the number of real places of $K$, and $r_2$ is the number of complex places of $K$.

For a rational prime number $p$, we write the prime ideal factorization of $p$ in $K$ as 
\[p=\mathfrak{p}_1^{e(\mathfrak{p}_1,p)}
\cdots\mathfrak{p}_{g(p)}^{e(\mathfrak{p}_{g(p)},p)}\]
with $N(\mathfrak{p}_j)=p^{f(\mathfrak{p}_j,p)}$ for $j=1, \ldots, g(p)$. 
Then the product in \eqref{eq2.1} is written as 
\[\zeta_K(s)=\prod_p\prod_{j=1}^{g(p)}(1-p^{-f(\mathfrak{p}_j,p)s})^{-1}.\]
Hence we have 
\begin{equation}\label{eq2.2}
\frac{\zeta'_K}{\zeta_K}(s)
=-\sum_{p}\sum_{j=1}^{g(p)}\sum_{k=1}^\infty
\frac{f(\mathfrak{p}_j,p)\log p}{p^{kf(\mathfrak{p}_j,p)s}}
=-\sum_{n=1}^\infty\frac{\Lambda_K(n)}{n^s}
\end{equation}
in the half-plane $\sigma>1$, where the function $\Lambda_K(n)$ is supported on prime powers and 
\begin{equation}\label{eq2.3}
\Lambda_K(p^m)
=\log p\sum_{\stackrel{1\leq j\leq{g(p)}}{f(\mathfrak{p}_j,p)|m}}f(\mathfrak{p}_j,p).
\end{equation}
Due to the fundamental equation  
\[e(\mathfrak{p}_1,p)f(\mathfrak{p}_1,p)+\cdots
+e(\mathfrak{p}_{g(p)},p)f(\mathfrak{p}_{g(p)},p)=d,\] 
the function $\Lambda_K(n)$ is bounded by $d\Lambda(n)$, where $\Lambda(n)=\Lambda_\mathbb{Q}(n)$ denotes the usual von Mangoldt function. 
This upper bound for $\Lambda_K(n)$ is an essential tool in the following study. 

\subsection{Notations and statements of results}\label{sec2.2}
Let 
\[D_\epsilon=\begin{cases}
\frac{1}{2} &\text{if}~ d=1,~2, \\
1-\frac{1}{d+\epsilon} &\text{if}~ d\geq3 \end{cases}\]
for any $\epsilon>0$, 
and let $\psi_w(z)$ denote the additive character as in \eqref{eq1.4}.
We define 
\[m_K(u,v,\sigma)=\prod_pm_{K,p}(u,v,\sigma)\]
for $\sigma>1/2$ and $u,v\in\mathbb{R}$, where
\begin{equation}\label{eq2.4}
m_{K,p}(u,v,\sigma)=\int_0^1\psi_w\Big(-\sum_{m=1}^\infty
\frac{\Lambda_K(p^m)}{p^{m\sigma}}e^{2\pi imt}\Big)dt
\end{equation}
for $w=u+iv$.
At first, we prove the following result using the function $m_K(u,v,\sigma)$.
\begin{prop}\label{prop2.1}
Let $\sigma_1$ be a large fixed positive real number.
Let $\theta,\delta>0$ with $\delta+3\theta<1/2$.
Then for each $\epsilon>0$, there exists $T_0=T_0(\theta,\delta,\sigma_1,\epsilon,K)>0$ such that for all $T\geq T_0$,  for all $\sigma\in[D_\epsilon+(\log T)^{-\theta},\sigma_1]$, and for all $w=u+iv$ whose absolute values are not greater than $(\log T)^\delta$, we have 
\begin{align}\label{eq2.5}
\frac{1}{T}\int_0^T\psi_w\Big(\frac{\zeta'_K}{\zeta_K}(\sigma+it)\Big)dt
=~&m_K(u,v,\sigma) \\
&+O_{K,\sigma_1}\Big(\exp\big(-\frac{1}{4}(\log T)^{\frac{2}{3}\theta}\big)\Big). \nonumber
\end{align}
The implied constant depends only on $K$ and $\sigma_1$. 
\end{prop}

The main result of the present paper is stated by using the Fourier inverse transformation of $m_K(u,v,\sigma)$. 
Thus we next consider the estimates on $m_K(u,v,\sigma)$. 
\begin{prop}\label{prop2.2}
Let $\sigma_1$ be a large fixed positive real number, and suppose $\sigma\in(1/2,\sigma_1]$.
Then there exists $J_K(\sigma)>0$ such that for all $u,v\in\mathbb{R}$ with $|u|+|v|\geq J_K(\sigma)$ and for all non-negative integers $n$ and $m$, we have
\[\frac{\partial^{n+m}}{\partial u^n\partial v^m}m_{K}(u,v,\sigma)
\ll_{n,m}\exp\Big(-\frac{c_K(\sigma_1)}{2\sigma-1}(|u|+|v|)^{\frac{1}{\sigma}}
(\log(|u|+|v|))^{\frac{1}{\sigma}-1}\Big).\]
Here the constant $J_K(\sigma)$ depends only on $K$ and $\sigma$, and the positive  constant  $c_K(\sigma_1)$ depends only on $K$ and $\sigma_1$. 
The implied constant depends only on $n$ and $m$.
\end{prop}

\vspace{3mm}

To state the main result, we prepare the Fourier transformations and the class of test functions.
The measures $|dz|$ and $|dw|$ denote $(2\pi)^{-1}dxdy$ and $(2\pi)^{-1}dudv$, where $z=x+iy$ and $w=u+iv$, respectively. 
For an integrable function $f(z)$ on $\mathbb{C}$, we define the Fourier transformation $f^{\wedge}(w)$ and the Fourier inverse transformation $f^{\vee}(w)$ as
\[f^{\wedge}(w)=\int_\mathbb{C}f(z)\psi_w(z)|dz|\]
and
\[f^{\vee}(w)=\int_\mathbb{C}f(z)\psi_{-w}(z)|dz|.\]
We usually write $f^{\wedge}$ as $\tilde{f}$. 
Then, according to the definition in \cite{IharaMatsumoto2}, let $\Lambda$ be the set of all functions $f$ such that $f$ and $f^{\wedge}$ belong to $L^1\cap L^\infty$, and the formula $(f^{\wedge})^{\vee}=f$ holds.
We see that the Schwartz class on $\mathbb{C}$ is included in the class $\Lambda$, hence any infinitely differentiable and compactly supported function belongs to $\Lambda$. 

By Proposition \ref{prop2.2}, $m_K(u,v,\sigma)$ belongs to $\Lambda$ as a function of $w=u+iv$ for $u,~v\in\mathbb{R}$. 
Therefore, we define the function $M_{K,\sigma}(z)$ as the Fourier inverse transformation of $m_K(u,v,\sigma)$, that is 
\[M_{K,\sigma}(z)=\int_\mathbb{C}m_K(u,v,\sigma)\psi_{-z}(w)|dw|\]
for $\sigma>1/2$.
Moreover, we have also $\tilde{M}_{K,\sigma}(w)=m_K(u,v,\sigma)$ for $w=u+iv$. 

We finally state the main result of the paper. 
\begin{thm}\label{thm2.3}
Let $\sigma_1$ be a large fixed positive real number.
Let $\theta$ and $\delta$ be two positive real numbers with $\delta+3\theta<1/2$.
Then for each $\epsilon$, there exists a positive real number $T_0=T_0(\theta,\delta,\sigma_1,\epsilon,K)$ such that for all $T\geq T_0$ and for all $\sigma$ in the interval $[D_\epsilon+(\log T)^{-\theta},\sigma_1]$, we have 
\[\frac{1}{T}\int_0^T \Phi\Big(\frac{\zeta'_K}{\zeta_K}(\sigma+it)\Big)dt
=\int_{\mathbb{C}}\Phi(z)M_{K,\sigma}(z)|dz|+E\]
for any $\Phi\in\Lambda$, where $E$ is 
\[\ll_{K,\sigma_1}\exp\big(-\frac{1}{4}(\log T)^{\frac{2}{3}\theta}\big)
\int_{|w|\leq(\log T)^\delta}|\tilde{\Phi}(w)|~|dw|+\int_{|w|>(\log T)^\delta}|
\tilde{\Phi}(w)|~|dw|.\]
The implied constant depends only on $K$ and $\sigma_1$. 
\end{thm}

\vspace{3mm}

\subsection{Remarks on the relationship to the other methods}\label{subsec2.3}
We first remark on the relation between results of this paper and results of Ihara--Matsumoto. 
Theorem \ref{thm2.3} looks generalization of Ihara's result in Section \ref{sec1} for the case (C). 
On the other hand, Ihara {\cite[see (1.4.3), (1.5.2), and (3.1.2)]{Ihara}} has introduced the function $\tilde{M}_\sigma^{(K)}(w)$ defined by 
\[\tilde{M}_\sigma^{(K)}(w)=\prod_\mathfrak{p}
\int_0^1\psi_w\Big(-\sum_{m=1}^\infty
\frac{\log N(\mathfrak{p})}{N(\mathfrak{p})^{m\sigma}}e^{2\pi imt}\Big)dt\]
for any global field $K$.
Then the formula \eqref{eq1.3} holds in the three cases described in Section \ref{sec1} for $M_\sigma^{(K)}(z)$ defined by 
\[M_\sigma^{(K)}(z)=\int_\mathbb{C}\tilde{M}_\sigma^{(K)}(w)\psi_{-z}(w)|dw|.\]
It is notable that if $K=\mathbb{Q}$, $M_{\sigma}^{(\mathbb{Q})}(z)$ is a common density function applied in both cases (A) and (C). 
Otherwise, Ihara noted that $M_\sigma^{(K)}(z)$, for which the formula \eqref{eq1.3} holds in the case (A) or (B), is no longer applied in the case (C) {\cite[\S 1.5]{Ihara}}. 
We can prove the function $\tilde{M}_{K,\sigma}(w)=m_K(u,v,\sigma)$ in Proposition \ref{prop2.1} agrees with $\tilde{M}_{\sigma}^{(K)}(w)$ for $K=\mathbb{Q}$, however, if $K\neq\mathbb{Q}$, these functions do not coincide. 
As a result, some arguments being valid in the method of Ihara--Matsumoto are impossible, for instance, we know fewer structure of the function $M_{K,\sigma}(z)$ than $\tilde{M}_{K,\sigma}(w)$.

We next consider the relation to generalizations of Bohr--Jessen type research \cite{Matsumoto3}. 
For any fixed positive integer $N$, we define 
\[\zeta_{K}(s;N)=\prod_{p\in P_N}\prod_{i=1}^{g(p)}(1-p^{-f(\mathfrak{p}_i,p)s})^{-1},\]
where $P_N$ denotes the set of rational primes until the $N$-th. 
Let $V_N(T,R)$ be the Lebesgue measure of the set of all $t\in[0,T]$ for which $\log\zeta_{K}(\sigma+it;N)$ belongs to $R$. 
In \cite{Matsumoto3}, the proof of the existence of the limit \eqref{eq1.1} start with showing the limit 
\[W_N(R)=\lim_{T\to\infty}\frac{1}{T}V_N(T,R)\]
exists. 
We extend $W_N$ to a probabilistic measure on $\mathbb{C}$, and let $\Lambda(w;N)$ be its Fourier transform. 
This is calculated as 
\[\Lambda(w;N)=\prod_{p\in P_N}\int_0^1\psi_w\Big(-\sum_{i=1}^{g(p)}
\log(1-p^{-f(\mathfrak{p}_i,p)\sigma}e^{2\pi if(\mathfrak{p}_i,p)t})\Big)dt.\] 
If the following estimate 
\[\Lambda(w;N)=O(|w|^{-(2+\eta)})\]
holds uniformly in $N$ for some $\eta>0$, we can write $W(R)$ as in \eqref{eq1.2} due to L\'{e}vy's inversion formula. 
In order to prove this estimate, the result of Jessen and Wintner {\cite[Theorem 10]{JessenWintner}} has been used, which we only apply to the case the Euler product is convex in the meaning in \cite{Matsumoto2}. 
This is why the generalization of the integral formula \eqref{eq1.2} is restricted. 
Now, except for the difference between $\log\zeta_K(s)$ and $(\zeta'_K/\zeta_K)(s)$, these functions $\Lambda(w)$ and $\Lambda(w;N)$ coincide with the function $\tilde{M}_{K,\sigma}(w)$ and $\tilde{M}_{K,\sigma}(w;N)$, respectively, where $\tilde{M}_{K,\sigma}(w;N)$ is defined by a suitable partial product of \eqref{eq2.1}.
In this paper, we consider only the estimate of $\tilde{M}_{K,\sigma}(w)$ without using any results of Jessen--Wintner. 
This estimate does not tell of $\tilde{M}_{K,\sigma}(w;N)$ at all. 

Recently, Matsumoto and Umegaki \cite{MatsumotoUmegaki2} considered an alternative method for the estimate of $\Lambda(w;N)$, and they succeeded to prove the limit formula \eqref{eq1.2} for automorphic $L$-functions. 
This new method may be applied to prove the limit formula \eqref{eq1.2} for various zeta or $L$-functions with non-convex Euler product. 
On the other hand, the generalization of Ihara-Matsumoto theorem remains to be considered.

\section{Proof of Proposition \ref{prop2.1}}\label{sec3}
In this section, we prove Proposition \ref{prop2.1}. 
At first, we approximate the logarithmic derivative $(\zeta'_K/\zeta_K)(s)$ by some Dirichlet polynomials and replace the left hand side of \eqref{eq2.5} with 
\begin{equation}\label{eq3.1}
\frac{1}{R}\int_0^R\psi_w\Big(-\sum_{p\leq x^2}
\sum_{m=1}^\infty\frac{\Lambda_K(p^m)}{p^{m(\sigma+it)}}\Big)dt
\end{equation}
under an acceptable error, where $R$ is any positive real number greater than $T$, and $x$ is a certain function of $T$ that tends to infinity as $T$ tends to infinity. 
Next, we consider the limit of the integral \eqref{eq3.1} as $R$ tends to infinity. 
This limit will be written as 
\begin{equation}\label{eq3.2}
\prod_{p\leq x^2}m_{K,p}(u,v,\sigma),
\end{equation}
where $m_{K,p}(u,v,\sigma)$ are functions in \eqref{eq2.4} and $w=u+iv$. 
Finally, we achieve the proof of Proposition \ref{prop2.1} by estimating the difference between the function \eqref{eq3.2} and $m_K(u,v,\sigma)$. 

\subsection{Approximation by Dirichlet polynomials}
In this subsection, we sometimes assume $\sigma>D_\epsilon$. 
We define 
\[\zeta_{K}(s;x)=\prod_{p\leq x^2}\prod_{j=1}^{g(p)}(1-p^{-f(\mathfrak{p}_j,p)s})^{-1}\]
for any $s\in\mathbb{C}$ and $x>0$. 
The goal of this subsection is the following lemma. 
\begin{lem}\label{lem3.1}
Let $\sigma_1$ be a large fixed positive real number.
Let $\theta,\delta>0$ with $\delta+\theta<1/2$.
Then for each $\epsilon>0$, there exists a positive real number $T_0=T_0(\theta,\delta,\sigma_1,\epsilon,K)$ such that for all $R\geq T\geq T_0$, for all $\sigma$ in the interval $[D_\epsilon+(\log T)^{-\theta},\sigma_1]$, and for all $|w|\leq(\log T)^\delta$, we have 
\begin{equation}\label{eq3.3}
\frac{1}{T}\int_0^T\psi_w\Big(\frac{\zeta'_K}{\zeta_K}(\sigma+it)\Big)dt
=\frac{1}{R}\int_0^R\psi_w\Big(\frac{\zeta'_{K}}{\zeta_{K}}(\sigma+it;x)\Big)dt+E,
\end{equation}
where 
\[E\ll_{K,\sigma_1}\exp\big(-\frac{1}{4}(\log T)^{\frac{2}{3}\theta}\big),\]
and $x=\exp((\log T)^{\frac{5}{3}\theta})$. 
The implied constant depends only on $K$ and $\sigma_1$. 
\end{lem}

\vspace{3mm}

The proof of this lemma consists of four steps.
At first, we replace $(\zeta'_K/\zeta_K)(\sigma+it)$ in the left hand side of 
\eqref{eq3.3} with the Dirichlet polynomials 
\begin{equation}\label{eq3.4}
-\sum_{n\leq x^2}\frac{\Lambda_{K,x}(n)}{n^{\sigma+it}},
\end{equation}
where 
\[\Lambda_{K,x}(n)=
\begin{cases}
\Lambda_{K}(n)&{\rm if}~1\leq n\leq x,\\
\frac{\log(x^2/n)}{\log x}\Lambda_{K}(n)&{\rm if}~x\leq n\leq x^2.
\end{cases}\]
Next, we consider the differences between $\Lambda_K(n)$ and $\Lambda_{K,x}(n)$, and replace the Dirichlet polynomials \eqref{eq3.4} with  
\begin{equation}\label{eq3.5}
-\sum_{n\leq x^2}\frac{\Lambda_{K}(n)}{n^{\sigma+it}}.
\end{equation}
The third step looks different from the other ones. 
We replace the interval of the integrals $[0,T]$ with $[0,R]$ for any $R\geq T$.  
At last, we again replace the Dirichlet polynomials \eqref{eq3.5} with 
\[\frac{\zeta'_K}{\zeta_K}(\sigma+it;x)
=-\sum_{p\leq x^2}\sum_{m=1}^\infty\frac{\Lambda_{K}(p^m)}{p^{m(\sigma+it)}}.\]
Then the integral which we consider is the one in the right hand side of \eqref{eq3.3}.

\vspace{6mm}

\noindent{\it Step1}\hspace{3mm}
We prove the following lemma. 
\begin{lem}\label{lem3.2}
Let $\sigma_1$ be a large positive real number and $A>0$. 
Then for each $\epsilon>0$, there exist an absolute constant $T_0>0$ such that for all $T\geq T_0$, for all $\sigma\in(D_\epsilon,\sigma_1]$, and for all $|w|\leq A$, we have
\[\frac{1}{T}\int_0^T\psi_w\Big(\frac{\zeta'_K}{\zeta_K}(\sigma+it)\Big)dt
=\frac{1}{T}\int_0^T\psi_w\Big(-\sum_{n\leq x^2}
\frac{\Lambda_{K,x}(n)}{n^{\sigma+it}}\Big)dt+E_1,\]
where  
\begin{align*}
E_1\ll_{K,\sigma_1}~\frac{|w|}{\log x}&\Big(\frac{x}{T}+x^{-1-\sigma}
+\frac{x\log h\log T}{h}+\frac{x^{\frac{1}{2}(D_\epsilon-\sigma)}\log T}
{(\sigma-D_\epsilon)^2}\Big)\nonumber\\
&+hT^{-c(\sigma-D_\epsilon)}(\log T)^C+\frac{1}{T}, 
\end{align*}
and $h$ and $x$ are any functions of $T$ that tend to infinity with $T$.
The constants $c$ and $C$ are positive and depend only on $K$ and $\epsilon$, and the implied constant depends only on $K$ and $\sigma_1$.
\end{lem}

\vspace{3mm}

If $\sigma>1$, we can directly approximate $(\zeta'_K/\zeta_K)(s)$ from the Dirichlet series  \eqref{eq2.2}. 
If $\sigma\leq 1$, we connect $(\zeta'_K/\zeta_K)(s)$ to Dirichlet polynomials by supplementing terms that come from zeros and poles of zeta function. 
The following lemma is an analogue of the result of Selberg {\cite[Lemma 2]{Selberg1}}. 
\begin{lem}\label{lem3.3}
Let $x>1$. Then for any $s\in\mathbb{C}$ excluding zeros or the pole of $\zeta_K(s)$, we have
\begin{align}\label{eq3.6}
&\frac{\zeta'_K}{\zeta_K}(s)
=-\sum_{n\leq x^2}\frac{\Lambda_{K,x}(n)}{n^s}
-\frac{x^{1-s}-x^{2(1-s)}}{(1-s)^2\log x}\\
&+\frac{1}{\log x}\sum_{\substack{q:~\zeta_K(-q)=0\\\text{trivial}}}
\frac{x^{-q-s}-x^{2(-q-s)}}{(q+s)^2}
+\frac{1}{\log x}\sum_{\substack{\rho:~\zeta_K(\rho)=0\\\text{non-trivial}}}
\frac{x^{\rho-s}-x^{2(\rho-s)}}{(\rho-s)^2}.\nonumber
\end{align}
\end{lem}
\begin{proof}
Let $\alpha=\max\{2,1+\sigma\}$. 
We consider the integral 
\begin{equation}\label{eq3.7}
\frac{1}{2\pi i}\int_{\alpha-iT}^{\alpha+iT}\frac{x^{z-s}-x^{2(z-s)}}{(z-s)^2}
\frac{\zeta'_K}{\zeta_K}(z)dz.
\end{equation}
This integral converges to 
\begin{equation}\label{eq3.8}
\log x\sum_{n\leq x^2}\frac{\Lambda_{K,x}(n)}{n^s}
\end{equation}
as $T$ tends to infinity, due to \eqref{eq2.2} and the well known formula 
\[\frac{1}{2\pi i}\int_{c-i\infty}^{c+i\infty}\frac{y^z}{z^2}dz=
\begin{cases}
\log y&{\rm if}~y\geq1,\\
0&{\rm if}~0<y\leq1
\end{cases}\]
for any $c>0$.
We replace the vertical line of the integral in \eqref{eq3.7} by the other three sides of the rectangle with vertices at 
\[\alpha-iT,~~~\alpha+iT,~~~-U-iT,~~~-U+iT,\] 
where $U=\lfloor T\rfloor+\frac{1}{2}$. 
Here, as usual, the function $\lfloor x\rfloor$ denote the greatest integer less than or equal to $x$. 
We first consider the integrals 
\begin{equation}\label{eq3.9}
\frac{1}{2\pi i}\int_{-U\pm iT}^{-1\pm iT}\frac{x^{z-s}-x^{2(z-s)}}{(z-s)^2}
\frac{\zeta'_K}{\zeta_K}(z)dz
\end{equation}
and
\begin{equation}\label{eq3.10}
\frac{1}{2\pi i}\int_{-U-iT}^{-U+iT}\frac{x^{z-s}-x^{2(z-s)}}{(z-s)^2}
\frac{\zeta'_K}{\zeta_K}(z)dz.
\end{equation}
On these contours, we have 
\[\frac{\zeta'_K}{\zeta_K}(s)\ll_K\log|s|\]
if $T\geq1$ by {\cite[Satz 194]{Landau}}. 
Then we see the integrals \eqref{eq3.9} and \eqref{eq3.10} tend to zero as $T$ tends to infinity. 
The reminder work is estimating the integral 
\begin{equation}\label{eq3.11}
\frac{1}{2\pi i}\int_{-1\pm iT}^{\alpha\pm iT}\frac{x^{z-s}-x^{2(z-s)}}{(z-s)^2}
\frac{\zeta'_K}{\zeta_K}(z)dz.
\end{equation}
We also have 
\[\frac{\zeta'_K}{\zeta_K}(\sigma\pm iT)\ll_K\log T\]
on this horizontal contour if $T\geq2$ by {\cite[see (165), Satz 181, and Satz 182]{Landau}}, hence the integral \eqref{eq3.11} tends to zero as $T$ tends to infinity.

Therefore, the Dirichlet series \eqref{eq3.8} is equal to the sum of residues. 
The residue at $z=s$ is 
\[-\log x~\frac{\zeta'_K}{\zeta_K}(s); \]
the residue at $z=1$ is 
\[-\frac{x^{1-s}-x^{2(1-s)}}{(1-s)^2\log x}; \]
those at $z=-q$ and $z=\rho$ are 
\[\frac{x^{-q-s}-x^{2(-q-s)}}{(q+s)^2}
~~\text{and}~~
\frac{x^{\rho-s}-x^{2(\rho-s)}}{(\rho-s)^2},\]
respectively. Then the result follows. 
\end{proof}

In order to examine the term of non-trivial zeros in \eqref{eq3.6}, we prepare the following lemma.
\begin{lem}\label{lem3.4}
Let $N_K(\alpha,T)$ denote the number of zeros $\rho=\beta+i\gamma$ of $\zeta_K(s)$ with $\beta\geq\alpha$ and $0\leq\gamma\leq T$.
Then for any $\epsilon>0$, we have
\begin{equation}\label{eq3.12}
N_K(\alpha,T)\ll_KT^{1-c(\alpha-D_\epsilon)}(\log T)^C
\end{equation}
uniformly for $D_\epsilon\leq\alpha\leq1$, where $c$ and $C$ are positive constants that depend only on $K$ and $\epsilon$, and the implied constant depends only on $K$. 
\end{lem}
\begin{proof}
If $K=\mathbb{Q}$, the Dedekind zeta function is just the Riemann zeta function. 
In this case, there is a classical result of Selberg \cite{Selberg2}, which is  
\[N_\mathbb{Q}(\alpha,T)\ll T^{1-\frac{1}{4}(\alpha-\frac{1}{2})}\log T\]
uniformly for $1/2\leq\alpha\leq1$. 

If $d=[K:\mathbb{Q}]\geq3$, we have Heath--Brown's result \cite{HeathBrown1}: 
for any $\epsilon>0$, there exists a positive constant $C=C(\epsilon,K)$ such that 
\[N_K(\alpha,T)\ll_K T^{(d+\epsilon)(1-\alpha)}(\log T)^C\]
uniformly for $1/2\leq\alpha\leq1$. 
The implied constant depends only on $K$. 
We again obtain the estimate \eqref{eq3.12}, since 
\[(d+\epsilon)(1-\alpha)=1-(d+\epsilon)\Big\{\alpha-\big(1-\frac{1}{d+\epsilon}
\big)\Big\}
=1-(d+\epsilon)(\alpha-D_\epsilon).\]

The case of quadratic fields remains. 
The Dedekind zeta function of any quadratic field is the product of the Riemann zeta function and a Dirichlet $L$-function attached to a quadratic character. 
Fujii \cite{Fujii} showed the estimate 
\[N_\chi(\alpha,T)\ll T^{1-\frac{1}{4}(\alpha-\frac{1}{2})}\log T\]
uniformly for $1/2\leq\alpha\leq1$. 
Here $N_\chi(\alpha,T)$ denotes the number of zeros $\rho=\beta+i\gamma$ of the Dirichlet $L$-function $L(s,\chi)$ with $\beta\geq\alpha$ and $0\leq\gamma\leq T$. 
Therefore, $N_K(\alpha,T)$ satisfies the same estimate as the case $K=\mathbb{Q}$. 
\end{proof}

\begin{proof}[Proof of Lemma \ref{lem3.2}]
We approximate $(\zeta'_K/\zeta_K)(s)$ with the first term of \eqref{eq3.6} by estimating the reminder terms. 
In order to avoid the pole at $s=1$, we assume $t\geq2$. 
Furthermore, we also assume $t\notin\mathscr{B}_h$ for the purpose of the investigation of non-trivial zeros, where $\mathscr{B}_h=\mathscr{B}_h(\sigma,T)$ is the set of all $t\in[0,T]$ for which there exists a zero $\rho=\beta+i\gamma$ of $\zeta_K(s)$ such that $\beta\geq\tfrac{1}{2}(\sigma+D_\epsilon)$ and $|\gamma-t|\leq h$. 
Here, $h$ is a function of $T$ that tends to infinity as $T$ tends to infinity. 

Before proceeding to the approximation of $(\zeta'_K/\zeta_K)(s)$, we estimate the measure of $\mathscr{B}_h$. 
We have  
\[\text{meas}~\mathscr{B}_h\leq hN_K(\tfrac{1}{2}(\sigma+D_\epsilon),T), \]
where ``meas $\mathscr{B}_h$'' means the Lebesgue measure of $\mathscr{B}_h$.
Therefore, we obtain 
\[\text{meas}~\mathscr{B}_h\ll_KhT^{1-\frac{c}{2}(\sigma-D_\epsilon)}(\log T)^C\]
uniformly for $D_\epsilon\leq\sigma\leq\sigma_1$ by Lemma \ref{lem3.4}. 

We come back to the approximation of $(\zeta'_K/\zeta_K)(s)$. 
Let $t\in\mathscr{B}_h^c\cap[2,T]$. 
Then the second term of \eqref{eq3.6} is estimated as 
\[\ll\frac{x}{t^2\log x}\]
for $\sigma\in(1/2,\sigma_1]$.
The third term is 
\[\ll\frac{dx^{-\sigma}}{t^2\log x}+\frac{dx^{-1-\sigma}}{\log x}\]
for $\sigma\in(1/2,\sigma_1]$, since any trivial zero is located at non-positive integer, whose multiplicity is at most $d$. 
For the last term, we divide the set of non-trivial zeros into two sets. 
One of them consists of zeros $\rho=\beta+i\gamma$ with $|\gamma-t|\geq h$. 
If $\rho$ belongs to such a set, we have
\[\left|\frac{1}{\log x}\sum_{\substack{\rho=\beta+i\gamma\\|\gamma-t|\geq h}}
\frac{x^{\rho-s}-x^{2(\rho-s)}}{(\rho-s)^2}\right|
\leq\frac{2x}{\log x}\sum_{\substack{\rho=\beta+i\gamma\\|\gamma-t|\geq h}}
\frac{1}{|\gamma-t|^2}
\ll\frac{dx\log h\log T}{h\log x}\]
for $\sigma\in(1/2,\sigma_1]$ by {\cite[Satz 180]{Landau}}.
If $\rho$ belongs to the set of remaining non-trivial zeros, we have $\beta<(\sigma+D_\epsilon)/2$. 
We divide the reminder sum further into those zeros with $|\gamma-t|\geq1$ and those without. 
For the first sum, 
\begin{align}\label{eq3.13}
\left|\frac{1}{\log x}\sum_{\substack{\rho=\beta+i\gamma\\1\leq|\gamma-t|<h}}
\frac{x^{\rho-s}-x^{2(\rho-s)}}{(\rho-s)^2}\right|
&\leq\frac{2x^{\frac{1}{2}(D_\epsilon-\sigma)}}{\log x}
\sum_{\substack{\rho=\beta+i\gamma\\1\leq|\gamma-t|<h}}
\frac{1}{|\gamma-t|^2}\\
&\ll_K\frac{x^{\frac{1}{2}(D_\epsilon-\sigma)}\log T}{\log x}, \nonumber
\end{align}
and for the second sum, 
\begin{align}\label{eq3.14}
\left|\frac{1}{\log x}\sum_{\substack{\rho=\beta+i\gamma\\|\gamma-t|<1}}
\frac{x^{\rho-s}-x^{2(\rho-s)}}{(\rho-s)^2}\right|
&\leq\frac{2x^{\frac{1}{2}(D_\epsilon-\sigma)}}{(\frac{\sigma-D_\epsilon}{2})^2\log x}
\sum_{\substack{\rho=\beta+i\gamma\\|\gamma-t|<1}}1\\
&\ll_K\frac{x^{\frac{1}{2}(D_\epsilon-\sigma)}\log T}{(\sigma-D_\epsilon)^2\log x}\nonumber. 
\end{align}
The last inequalities in \eqref{eq3.13} and \eqref{eq3.14} are deduced from {\cite[Satz 181]{Landau}}.

To complete the proof, we see that the difference 
\[\frac{1}{T}\int_0^T\psi_w\Big(\frac{\zeta'_K}{\zeta_K}(\sigma+it)\Big)dt
-\frac{1}{T}\int_0^T\psi_w\Big(-\sum_{n\leq x^2}
\frac{\Lambda_{K,x}(n)}{n^{\sigma+it}}\Big)dt\]
is equal to 
\[\frac{1}{T}\int_{\mathscr{B}^c_h\cap[2,T]}
\Big\{\psi_w\Big(\frac{\zeta'_K}{\zeta_K}(\sigma+it)\Big)
-\psi_w\Big(-\sum_{n\leq x^2}\frac{\Lambda_{K,x}(n)}{n^{\sigma+it}}\Big)\Big\}dt
+O\Big(\frac{\text{meas}~\mathscr{B}_h+2}{T}\Big). \]
By the fact $|\psi_w(z_1)-\psi_w(z_2)|\ll|w|~|z_1-z_2|$ and the above estimates, we have the desired result. 
\end{proof}

\vspace{3mm}

\noindent{\it Step2}\hspace{3mm}
The goal of this step is proving the following lemma. 
\begin{lem}\label{lem3.5}
Let $\sigma_1$ be a large positive real number and $A>0$. 
Then there exist an absolute constant $T_0>0$ such that for all $T\geq T_0$, for all $\sigma\in(1/2,\sigma_1]$, and for all $|w|\leq A$, we have 
\[\frac{1}{T}\int_0^T\psi_w\Big(-\sum_{n\leq x^2}
\frac{\Lambda_{K,x}(n)}{n^{\sigma+it}}\Big)dt
=\frac{1}{T}\int_0^T\psi_w\Big(-\sum_{n\leq x^2}
\frac{\Lambda_{K}(n)}{n^{\sigma+it}}\Big)dt+E_2, \]
where
\[E_2\ll\frac{d|w|\log x}{T^{\frac{1}{2}}(\sigma-\frac{1}{2})^{\frac{1}{2}}}
(T+x^2)^{\frac{1}{2}}x^{\frac{1}{2}-\sigma},\]
and $x$ is any function of $T$ that tends to infinity with $T$. 
The implied constant is absolute. 
\end{lem}
\begin{proof}
By using the estimate $|\psi_w(z_1)-\psi_w(z_2)|\ll|w|~|z_1-z_2|$, the difference 
\[\frac{1}{T}\int_0^T\psi_w\Big(-\sum_{n\leq x^2}
\frac{\Lambda_{K,x}(n)}{n^{\sigma+it}}\Big)dt
-\frac{1}{T}\int_0^T\psi_w\Big(-\sum_{n\leq x^2}
\frac{\Lambda_{K}(n)}{n^{\sigma+it}}\Big)dt.\]
is estimated as 
\begin{align}\label{eq3.15}
\ll\frac{|w|}{T}\int_0^T \Big|\sum_{n\leq x^2}
\frac{\Lambda_{K,x}(n)-\Lambda_{K}(n)}{n^{\sigma+it}}\Big|dt
&\leq\frac{|w|}{T}\int_0^T \Big|\sum_{x\leq n\leq x^2}
\frac{\Lambda_K(n)}{n^{\sigma+it}}\Big|dt \nonumber\\
&\leq\frac{|w|}{T^{\frac{1}{2}}}\Big\{\int_0^T \Big|\sum_{x\leq n\leq x^2}
\frac{\Lambda_K(n)}{n^{\sigma+it}}\Big|^2dt\Big\}^{\frac{1}{2}}.
\end{align}
The last inequality is due to Cauchy's inequality. 
Then we apply the following result of {\cite[Theorem 6.1]{Montgomery}}: for any real $T_0$ and $T$, we have 
\begin{equation}\label{eq3.16}
\int_{T_0}^{T_0+T}\Big|\sum_{n=1}^Na_nn^{-it}\Big|^2dt
=\Big(T+\theta\frac{4\pi}{\sqrt3}N\Big)\sum_{n=1}^N|a_n|^2,
\end{equation}
where $-1\leq\theta\leq1$.
Hence \eqref{eq3.15} is 
\[\ll\frac{d|w|\log x}{T^{\frac{1}{2}}(\sigma-\frac{1}{2})^{\frac{1}{2}}}
(T+x^2)^{\frac{1}{2}}x^{\frac{1}{2}-\sigma},\]
here we use the bound $\Lambda_K(n)\leq d\Lambda(n) \leq d\log n$. 
This proves Lemma \ref{lem3.5}. 
\end{proof}

\vspace{3mm}

\noindent{\it Step3}\hspace{3mm}
At the third step, we prove the following lemma. 
\begin{lem}\label{lem3.6}
Let $\sigma_1$ be a large fixed positive real number and $A>0$. 
Then there exists an absolute constant $T_0>0$ such that for all $R\geq T\geq T_0$, for all $\sigma\in(1/2,\sigma_1]$, and for all $|w|\leq A$, 
we have
\[\frac{1}{T}\int_0^T\psi_w\Big(-\sum_{n\leq x^2}
\frac{\Lambda_{K}(n)}{n^{\sigma+it}}\Big)dt
=\frac{1}{R}\int_0^R\psi_w\Big(-\sum_{n\leq x^2}
\frac{\Lambda_{K}(n)}{n^{\sigma+it}}\Big)dt+E_3,\]
where 
\begin{align*}
E_3\ll&\frac{d^Nx^{5N}}{T}(1+|w|^2)^{\frac{N}{2}}\\
&+\frac{(8d|w|)^N}{N!}\Big(1+\frac{x^N}{T}\Big)
\{(\zeta(2\sigma)^{\frac{1}{2}}\log x)^N(\tfrac{N}{2})!+\zeta'(2\sigma)^N\},
\end{align*}
and $x$ and $N$ are functions that grows with $T$, $N$ being an even integer. 
The implied constant is absolute.
\end{lem}

We use the following result in the proof. 
\begin{lem}[\cite{Ghosh}]\label{lem3.7}
Let $p_i$ and $q_j$ denote prime numbers and set $y\geq y_0$.
Let $\tau$ be a positive number, and suppose that the $\delta(n)$ are complex numbers satisfying $|\delta(n)|\leq c$ for some fixed constant $c>0$. 
Then for any integer $k\geq1$, 
\[\sum_{\substack{p_1,\ldots,p_k<y\\q_1,\ldots,q_k<y\\
p_1\cdots p_k=q_1\cdots q_k}}
\frac{\delta(p_1)\cdots\delta(p_k)\delta(q_1)\cdots\delta(q_k)}{(p_1\cdots p_k)^\tau}\]
is equal to 
\[k!\Big(\sum_{p<y}\frac{\delta^2(p)}{p^{2\tau}}\Big)^k
+O\Big(c^{2k}k!\Big(\sum_{p<y}p^{-2\tau}\Big)^{k-2}
\Big(\sum_{p<y}p^{-4\tau}\Big)\Big).\]
The implied constant is absolute. 
\end{lem}

\begin{proof}[Proof of Lemma \ref{lem3.6}]
To simplify notation, let $F_K(t)=F_K(t,\sigma;x)$ denotes the function 
\[-\sum_{n\leq x^2}\frac{\Lambda_{K}(n)}{n^{\sigma+it}}.\]
We have  
\[\psi_w(z)=\sum_{n=0}^{N-1}\frac{i^n}{n!} \langle z,w\rangle^n
+O\Big(\frac{1}{N!}\langle z,w\rangle^N\Big),\]
by the Taylor series expansion, where $N$ is any function of $T$ that takes even values and increases with $T$. 
Then the integral
\[\frac{1}{R}\int_0^R\psi_w(F_K(t))dt\]
is equal to 
\begin{equation}\label{eq3.17}
\sum_{n=0}^{N-1}\frac{i^n}{n!}\frac{1}{R}\int_0^R\langle F_K(t),w\rangle^ndt
+O\Big(\frac{1}{N!}\frac{1}{R}\int_0^R\Big|\langle F_K(t),w\rangle\Big|^Ndt\Big).
\end{equation}
By using the expansion 
\begin{align*}
\langle F_K(t),w\rangle^n
&=\Big(\frac{\overline{w}F_K(t)+w\overline{F_K(t)}}{2}\Big)^n\\
&=\frac{1}{2^n}\sum_{m=0}^n\binom{n}{m}
\overline{w}^mw^{n-m}F_K(t)^m\overline{F_K(t)}^{n-m},
\end{align*}
the first term of \eqref{eq3.17} is equal to 
\[\sum_{n=0}^{N-1}\frac{i^n}{2^nn!}\sum_{m=0}^n\binom{n}{m}\overline{w}^mw^{n-m}
\frac{1}{R}\int_0^RF_K(t)^m\overline{F_K(t)}^{n-m}dt.\]
Therefore, the difference we should consider is equal to 
\begin{align}\label{eq3.18}
&\sum_{n=0}^{N-1}\frac{i^n}{2^nn!}\sum_{m=0}^n\binom{n}{m}\overline{w}^mw^{n-m}
H(n,m)\\
&+O\Big(\frac{|w|^N}{N!}\Big(\frac{1}{T}\int_0^T|F_K(t)|^Ndt
+\frac{1}{R}\int_0^R|F_K(t)|^Ndt\Big)\Big), \nonumber
\end{align}
where 
\[H(n,m)=\frac{1}{T}\int_0^TF_K(t)^m\overline{F_K(t)}^{n-m}dt
-\frac{1}{R}\int_0^RF_K(t)^m\overline{F_K(t)}^{n-m}dt.\]
Then we investigate $H(n,m)$. 
Let $l=n-m$.
At first, suppose $m$ and $l$ are both non-zero. 
Then we have 
\begin{align*}
&F_K(t)^m\overline{F_K(t)}^l\\
&=\sum_{\substack{n_{11},\ldots,n_{1m}\leq x^2\\n_{21},\ldots,n_{2l}\leq x^2}}
\frac{\Lambda_K(n_{11})\cdots\Lambda_K(n_{1m})\Lambda_K(n_{21})\cdots
\Lambda_K(n_{2l})}
{(n_{11}\cdots n_{1m}n_{21}\cdots n_{2l})^\sigma}
\Big(\frac{n_{11}\cdots n_{1m}}{n_{21}\cdots n_{2l}}\Big)^{-it}.
\end{align*}
Hence we obtain 
\begin{align*}
&H(n,m)\\
&\ll \frac{d^n}{T}\sum_{\substack{n_{11},\ldots,n_{1m}\leq x^2\\ 
n_{21},\ldots,n_{2l}\leq x^2\\ 
n_{11}\cdots n_{1m}\neq n_{21}\cdots n_{2l}}}
\frac{\Lambda(n_{11})\cdots\Lambda(n_{1m})\Lambda(n_{21})\cdots
\Lambda(n_{2l})}
{(n_{11}\cdots n_{1m}n_{21}\cdots n_{2l})^\sigma}
\left|\log\frac{n_{11}\cdots n_{1m}}{n_{21}\cdots n_{2l}}\right|^{-1},
\end{align*}
since $\Lambda_K(n)\leq d\Lambda(n)$.
Therefore, we conclude 
\[H(n,m)\ll\frac{d^nx^{5n}}{T}\]
by {\cite[Lemma 2.1.7]{Guo1}}. 
Next, we suppose $l=0$ and $m\neq0$, that is, $n=m>0$.
Then we have 
\[\frac{1}{R}\int_0^RF_K(t)^m\overline{F_K(t)}^{n-m}dt
=\frac{1}{R}\int_0^R\sum_{k_1,\ldots,k_n\leq x^2}
\frac{\Lambda_K(k_1)\cdots\Lambda_K(k_n)}{(k_1\cdots k_n)^{\sigma+it}}dt, \]
hence this integral is bounded as 
\[\ll\frac{d^n}{R}\sum_{k_1,\ldots,k_n\leq x^2}
\frac{\Lambda(k_1)\cdots\Lambda(k_n)}
{(k_1\cdots k_n)^\sigma|\log(k_1\cdots k_n)|}
\ll\frac{d^n}{T}\Big(\sum_{k\leq x^2}\Lambda(k)\Big)^n
\ll\frac{d^nx^{5n/2}}{T}. \]
Thus we have again $H(n,m)\ll T^{-1}d^nx^{5n}$. 
Similarly, we have $H(n,m)\ll T^{-1}d^nx^{5n}$ assuming $m=0$ and $l\neq0$, and we have $H(0,0)=0$. 
Therefore, the first term of \eqref{eq3.18} is 
\[\ll\frac{d^Nx^{5N}}{T}(1+|w|^2)^{\frac{N}{2}}.\]
For the second term of \eqref{eq3.18}, we must estimate the moments of $F_K(t)$. 
We divide the sum defining $F_K(t)$ as follows: 
\begin{equation}\label{eq3.19}
\sum_{n\leq x^2}\frac{\Lambda_{K}(n)}{n^{\sigma+it}}
=\sum_{p\leq x^2}\frac{\Lambda_{K}(p)}{p^{\sigma+it}}
+\sum_{\substack{p^a\leq x^2\\a\geq2}}\frac{\Lambda_{K}(p^a)}{p^{a(\sigma+it)}}.
\end{equation}
The absolute value of the second term of \eqref{eq3.19} is estimated as 
\[\leq d\sum_{p\leq x^2}\log p\sum_{a=2}^\infty\frac{1}{p^{a\sigma}}
\leq4d\sum_{p\leq x^2}\frac{\log p}{p^{2\sigma}}
\leq4d|\zeta'(2\sigma)|.\]
Let $N=2M$. Then we have 
\begin{align*}
\int_0^R|F_K(t)|^{2M}dt
&\leq\int_0^R\Big(\Big|\sum_{p\leq x^2}\frac{\Lambda_K(p)}{p^{\sigma+it}}\Big|
+4d|\zeta'(2\sigma)|\Big)^{2M}dt\\
&\leq4^M\int_0^R\Big|\sum_{p\leq x^2}\frac{\Lambda_K(p)}{p^{\sigma+it}}\Big|^{2M}dt
+4^M(4d\zeta'(2\sigma))^{2M}R. \nonumber
\end{align*}
Using the formula \eqref{eq3.16} again, the first term of the above is 
\begin{equation}\label{eq3.20}
\ll d^{2M}(R+x^{2M})\sum_{\substack{p_1,\ldots,p_M\leq x^2\\q_1,\ldots,q_M\leq x^2\\
p_1\cdots p_M=q_1\cdots q_M}}
\frac{\log p_1\cdots\log p_M\log q_1\cdots\log q_M}{(p_1\cdots p_M)^{2\sigma}}.
\end{equation}
By using Lemma \ref{lem3.7}, the Dirichlet polynomial \eqref{eq3.20} is estimated as
\begin{align*}
&\ll d^{2M}(R+x^{2M})\Big\{
M!\Big(\sum_{p<x^2}\frac{\log^2p}{p^{2\sigma}}\Big)^M\\
&\hspace{30mm}+(\log x)^{2M}M!\Big(\sum_{p<x^2}p^{-2\sigma}\Big)^{M-2}
\Big(\sum_{p<x^2}p^{-4\sigma}\Big)\Big\}\\
&\ll d^{2M}(R+x^{2M})M!(\log x)^{2M}\zeta(2\sigma)^M.
\end{align*}
Consequently, we obtain that 
\[\int_0^R|F_K(t)|^{2M}dt
\ll d^{2M}2^{6M}(R+x^{2M})\{M!(\log x)^{2M}
\zeta(2\sigma)^M+\zeta'(2\sigma)^{2M}\}.\]
Hence the second term of \eqref{eq3.18} is  
\[\ll\frac{(8d|w|)^N}{N!}\Big(1+\frac{x^N}{T}\Big)
\{(\zeta(2\sigma)^{\frac{1}{2}}\log x)^N(\tfrac{N}{2})!+\zeta'(2\sigma)^N\}.\]
Then Lemma \ref{lem3.6} follows.
\end{proof}

\vspace{3mm}

\noindent{\it Step4}\hspace{3mm}
At the final step, we prove the following lemma.
\begin{lem}\label{lem3.8}
Let $\sigma_1$ be a large fixed positive real number and $A>0$. 
Then there exists an absolute constant $T_0>0$ such that for all $R\geq T\geq T_0$, 
for all $1/2<\sigma\leq\sigma_1$, and for all $|w|\leq A$, 
we have 
\[\frac{1}{R}\int_0^R\psi_w\Big(-\sum_{n\leq x^2}
\frac{\Lambda_{K}(n)}{n^{\sigma+it}}\Big)dt
=\frac{1}{R}\int_0^R\psi_w\Big(\frac{\zeta'_K}{\zeta_K}(\sigma+it;x)
\Big)dt+E_4,\]
where 
\[E_4\ll\frac{d|w|(\log x)x^{1-2\sigma}}{\sigma-\frac{1}{2}},\]
and $x$ is any function of $T$ that grows with $T$. 
The implied constant is absolute.
\end{lem}
\begin{proof}
The difference 
\[\Big(-\sum_{n\leq x^2}\frac{\Lambda_{K}(n)}{n^{\sigma+it}}\Big)
-\frac{\zeta'_K}{\zeta_K}(\sigma+it;x) \]
is estimated as
\begin{equation}\label{eq3.21}
\ll d\sum_{p\leq x^2}\log p\sum_{p^m>x^2}\frac{1}{p^{m\sigma}}
=d\sum_{p\leq x}\log p\sum_{p^m>x^2}\frac{1}{p^{m\sigma}}
+d\sum_{x<p\leq x^2}\log p\sum_{p^m>x^2}\frac{1}{p^{m\sigma}}.
\end{equation}
For the first term of the right hand side of \eqref{eq3.21}, the inner sum is calculated as 
\[-\int_{\frac{2\log x}{\log p}}^\infty\Big(
\sum_{\frac{2\log x}{\log p}\leq m\leq t}1\Big)\frac{d}{dt}p^{-\sigma t}dt
\ll \Big(\frac{\log x}{\log p}+1\Big)x^{-2\sigma}.\]
Thus the first term is $\ll d(\log x) x^{1-2\sigma}$. 
The inner sum of the second term of \eqref{eq3.21} is $\ll p^{-2\sigma}$, hence the second term is 
\[\ll\frac{d|w|(\log x)x^{1-2\sigma}}{\sigma-\frac{1}{2}}.\]
Therefore, we complete the proof of Lemma \ref{lem3.8}. 
\end{proof}

\begin{proof}[Proof of Lemma \ref{lem3.1}]
We now choose the functions $x,h,A,$ and $N$ as follows. 
Assume that $\sigma$ is in the interval $[D_\epsilon+(\log T)^{-\theta},\sigma_1]$. 
Then we set 
\begin{equation}\label{eq3.22}
x=\exp((\log T)^{\theta_1}),~~h=\exp((\log T)^{\theta_2}),~~
A=(\log T)^\delta,~~N=2\lfloor(\log T)^{\theta_3}\rfloor, 
\end{equation}
where $\theta,~\theta_1,~\theta_2,~\theta_3,$ and $\delta$ are positive real numbers satisfying 
\begin{align}\label{eq3.23}
0<\delta<1,~~0<\theta<\theta_1<1-\theta<1,~~\theta_2=\tfrac{1}{2}(\theta_1+1-\theta),
\\
2\delta+\theta+3\theta_1<1,~~2(\delta+\tfrac{1}{2}\theta+\theta_1)<\theta_3<1-\theta_1. \nonumber
\end{align}
Under these choices of the functions, the errors $E_i~(i=1,2,3,4)$ of Lemmas \ref{lem3.2}, \ref{lem3.5}, \ref{lem3.6}, and \ref{lem3.8} are 
\[\ll_{K,\sigma_1}\exp\big(-\frac{1}{4}(\log T)^{\theta_1-\theta}\big). \]
The implied constant depends only on $K$ and $\sigma_1$.
Then Lemma \ref{lem3.1} is deduced from lemmas in the above four steps with $\theta_1=5\theta/3$.
\end{proof}

\subsection{Limit formula of a special integral}
In the previous subsection, we finally reached the integral \eqref{eq3.1}. 
Next, we prove that the limit of the integral can be written as a certain finite product. \begin{lem}\label{lem3.9}
For any $\sigma,x>0$, and for any $w=u+iv\in\mathbb{C}$,
\[\lim_{R\to \infty}\frac{1}{R}\int_0^R \psi_w\Big(\frac{\zeta'_K}{\zeta_K}
(\sigma+it;x)\Big)dt
=m_K(u,v,\sigma;x), \]
where the right hand side is defined by 
\[m_K(u,v,\sigma;x)=\prod_{p\leq x^2}m_{K,p}(u,v,\sigma)\]
with the function $m_{K,p}(u,v,\sigma)$ as in \eqref{eq2.4}.
\end{lem}

\vspace{3mm}

This is a simple consequence of the following result of Heath--Brown {\cite[Lemma 2]{HeathBrown2}}. 
\begin{lem}
Let $b_i(t)$ $(1\leq i\leq k)$ be continuous functions of period 1 from 
$\mathbb{R}$ to $\mathbb{C}$.
Then there exists the limit 
\[\lim_{R\to\infty}\frac{1}{R}\int_0^R\prod_{i\leq k}b_i(\gamma_i t)dt\]
for any real constants $\gamma_1,\ldots,\gamma_k$.
Moreover, if the numbers $\gamma_i$ are linearly independent over $\mathbb{Q}$,
then the limit is equal to 
\[\prod_{i\leq k}\int_0^1b_i(t)dt.\]
\end{lem}

\begin{proof}[Proof of Lemma \ref{lem3.9}]
We take $b_p(t)$ as the function
\[\psi_w\Big(-\sum_{m=1}^\infty\frac{\Lambda_K(p^m)}{p^{m\sigma}}e^{2\pi imt}\Big)\]
and $\gamma_p=-\log p/2\pi$, then Heath--Brown's lemma yields Lemma \ref{lem3.9}.
\end{proof}

\subsection{Behavior of $m_{K,p}(u,v,\sigma)$ for large $p$}
To prove Proposition \ref{prop2.1}, the remainder work is the replacement of $m_K(u,v,\sigma;x)$ with $m_K(u,v,\sigma)$. 
Keeping in mind that $|m_{K,p}(u,v,\sigma)|\leq1$, we show the following estimate. 
\begin{lem}\label{lem3.11}
Let $\sigma_1$ be a large fixed positive real number. 
Let $\delta, \theta_1$ and $\theta$ are real numbers satisfying the conditions \eqref{eq3.23} with $\theta_1=5\theta/3$. 
Then there exists a positive constant $T_0=T_0(\delta, \theta, \sigma_1)$ such that for all $T\geq T_0$, for all $1/2+(\log T)^{-\theta}<\sigma\leq\sigma_1$, and for all real numbers $u$ and $v$ in the interval $[-A,A]$, 
\[\Big|\prod_{p\geq x^2}m_{K,p}(u,v,\sigma)-1\Big|
\ll_{d,\sigma_1}\exp\big(-\frac{1}{4}(\log T)^{\frac{2}{3}\theta}\big),\]
where $x$ and $A$ are the functions as in \eqref{eq3.22}. 
The implied constant depends only on $d$ and $\sigma_1$. 
\end{lem}
\begin{proof}
We calculate $m_{K,p}(u,v,\sigma)$ in \eqref{eq2.4} as 
\begin{equation}\label{eq3.24}
m_{K,p}(u,v,\sigma)=\int_0^1\exp\{iua_p^{(1)}(t)+iva_p^{(2)}(t)\}dt,
\end{equation}
where 
\begin{equation}\label{eq3.25}
a_p^{(1)}(t)=-\sum_{m=1}^\infty\frac{\Lambda_K(p^m)}{p^{m\sigma}}\cos(2\pi mt)
\end{equation}
and
\begin{equation}\label{eq3.26}
a_p^{(2)}(t)=-\sum_{m=1}^\infty\frac{\Lambda_K(p^m)}{p^{m\sigma}}\sin(2\pi mt).
\end{equation}
We expand the exponential in \eqref{eq3.24}. 
Due to the identities 
\[\int_0^1a_p^{(1)}(t)dt=\int_0^1a_p^{(2)}(t)dt=0, \]
\[\int_0^1(a_p^{(1)}(t))^2dt=\int_0^1(a_p^{(2)}(t))^2dt
=\frac{1}{2}\sum_{m=1}^\infty\frac{\Lambda_K(p^m)^2}{p^{2m\sigma}}, \]
and
\[\int_0^1a_p^{(1)}(t)a_p^{(2)}(t)dt=0, \]
the function $m_{K,p}(u,v,\sigma)$ is written as 
\begin{equation}\label{eq3.27}
1-\mu+R, 
\end{equation}
where
\[\mu=\frac{u^2+v^2}{4}\sum_{m=1}^\infty
\frac{\Lambda_K(p^m)^2}{p^{2m\sigma}}\]
and 
\[R=\int_0^1\sum_{k=3}^\infty\frac{1}{k!}(iua_p^{(1)}(t)+iva_p^{(2)}(t))^kdt. \]
We assume $p\geq x^2$, where $x=\exp((\log T)^{\frac{5}{3}\theta})$. 
Then we find that $\mu$ and $R$ are small when $T$ is sufficiently large. 
In fact, we have 
\[|\mu|\ll\frac{(u^2+v^2)d^2\log^2p}{p^{2\sigma}}\ll(dAx^{1-2\sigma}\log x)^2\]
for $p\geq x^2$.
By the conditions for $x, A,$ and $\sigma$, 
\[Ax^{1-2\sigma}\log x\leq \exp\big(-\frac{1}{4}(\log T)^{\frac{2}{3}\theta}\big),\]
and this tends to zero as $T$ tends to infinity.
The argument for $R$ is similar. 

Hence we can define $\log m_{K,p}(u,v,\sigma)$, where $\log$ means the principal blanch of the logarithm. 
By using \eqref{eq3.27}, $\log m_{K,p}(u,v,\sigma)$ is calculated as $-\mu+R+O(|\mu|^2+R^2)$. 
Thus we have 
\[|\log m_{K,p}(u,v,\sigma)|\ll_d(u^2+v^2)\frac{\log^2p}{p^{2\sigma}},\]
where the implied constant depends only on $d$. 
This deduces the following estimate 
\begin{align*}
\prod_{p\geq x^2}m_{K,p}(u,v,\sigma)
&=\exp\Big(\sum_{p\geq x^2}\log m_{K,p}(u,v,\sigma)\Big)\\
&=1+O\Big(\sum_{p\geq x^2}|\log m_{K,p}(u,v,\sigma)|\Big)\\
&=1+O_d\Big((u^2+v^2)\sum_{p\geq x^2}\frac{\log^2p}{p^{2\sigma}}\Big), 
\end{align*}
and we have 
\[\sum_{p\geq x^2}\frac{\log^2p}{p^{2\sigma}}
\ll_{\sigma_1}\frac{x^{2(1-2\sigma)}\log x}{(\sigma-\frac{1}{2})^2}\]
for $\sigma>1/2$ by the Prime Number Theorem. 
After adopting the choices of the functions $x$ and $A$, we obtain the 
desired estimate. 
\end{proof}

\vspace{3mm}

\begin{proof}[Proof of Proposition \ref{prop2.1}]
Lemma \ref{lem3.11} yields 
\begin{equation}\label{eq3.28}
m_K(u,v,\sigma;x)=m_K(u,v,\sigma)
+O_{K,\sigma_1}(\exp\big(-\frac{1}{4}(\log T)^{\frac{2}{3}\theta}\big),
\end{equation}
since $|m_{K,p}(u,v,\sigma)|\leq1$.
Then Proposition \ref{prop2.1} is deduced from Lemmas \ref{lem3.1}, \ref{lem3.9}, and \eqref{eq3.28}.
\end{proof}

\section{Proof of Proposition \ref{prop2.2}}\label{sec4}
In this section, we regard $m_{K,p}(u,v,\sigma)$ in \eqref{eq3.24} as a function of complex variables $u$ and $v$. 
At first, we examine the differentiability of $m_K(u,v,\sigma)$. 
\begin{lem}\label{lem4.1}
Let $\sigma_1$ be a large fixed positive real number, and fix $\sigma\in(1/2,\sigma_1]$.
Then for any non-negative integers $n$ and $m$, the function 
\[\frac{\partial^{n+m}}{\partial u^n\partial v^m}m_K(u,v,\sigma)\]
is an entire function of $u$ for any fixed $v$, and it is an entire function of $v$ for any fixed $u$.
\end{lem}
\begin{proof}
We only consider the case of the first derivative with respect to $u$, since other cases are proved in a similar way. 
Fix a complex number $v$ arbitrarily.
We show that 
\begin{equation}\label{eq4.1}
\prod_{p<y}m_{K,p}(u,v,\sigma)
\end{equation}
converges uniformly in $\{u~;~|u|\leq R\}$ as $y\to\infty$ for any fixed $R>0$. 
As before we express $m_{K,p}(u,v,\sigma)$ as $1-\mu+R$ with
\[\mu=\frac{u^2+v^2}{4}\sum_{m=1}^\infty\frac{\Lambda_K(p^m)^2}{p^{2m\sigma}}\]
and
\[R=\int_0^1\sum_{k=3}^\infty\frac{1}{k!}(iua_p^{(1)}(t)+iva_p^{(2)}(t))^kdt.\]
We see that $\mu$ and $R$ are small for large $p$. 
In fact, $|\mu|$ is less than $1/4$ for $p>p_0$, since $|\mu|\ll_v(dR\log p~p^{-\sigma})^2$,  here $p_0$ does not depend on $u$.
The argument for $R$ is similar. 
Then we can define $\log m_{K,p}(u,v,\sigma)$ for $p>p_0$.
Hence we have 
\[|\log m_{K,p}(u,v,\sigma)|\leq Ad^2R^2\frac{\log^2p}{p^{2\sigma}},\]
here $A$ is a positive constant that depends only on $v$. 
The summation of the right hand side taken over $p$ converges, therefore the product  \eqref{eq4.1} uniformly converges for $u$ as $y$ tends to infinity. 
This proves that $m_K(u,v,\sigma)$ is an entire function of $u$.
\end{proof}

By Lemma \ref{lem4.1}, applying Cauchy's integral formula to $m_K(u,v,\sigma)$,  then Proposition \ref{prop2.2} comes down to the following lemma. 

\begin{lem}\label{lem4.2}
Let $\sigma_1$ be a large fixed positive real number, and suppose 
$\sigma\in(1/2,\sigma_1]$.
Then there exists a positive real number $J_K(\sigma)$ such that 
for all $\alpha,\beta\in\mathbb{R}$ with $|\alpha|+|\beta|\geq J_K(\sigma)$, we have
\[|m_{K}(u,v,\sigma)|
\leq\exp\Big(-\frac{c_K(\sigma_1)}{2\sigma-1}(|\alpha|+|\beta|)^{\frac{1}{\sigma}}
(\log(|\alpha|+|\beta|))^{\frac{1}{\sigma}-1}\Big)\]
for all $(u,v)\in D(\alpha,1/2)\times D(\beta,1/2)$,
where $D(c,r)$ is the disc on $\mathbb{C}$ with center $c$ and radius $r$. 
The constant $J_K(\sigma)$ depends only on $K$ and $\sigma$, and the constant $c_K(\sigma_1)>0$ depends only on $K$ and $\sigma_1$. 
\end{lem}

\vspace{3mm}

For given constants $\alpha$, $\beta$, and $c_0=c_0(K)$, let  
\[P_0=\Big(\frac{|\alpha|+|\beta|}{c_0}\log\frac{|\alpha|+|\beta|}{c_0}\Big)
^{\frac{1}{\sigma}},\]
and define
\[m_{K}^{(1)}(u,v,\sigma)=\prod_{p\leq P_0}m_{K,p}(u,v,\sigma)\]
and
\[m_{K}^{(2)}(u,v,\sigma)=\prod_{p>P_0}m_{K,p}(u,v,\sigma).\]

At first, we consider the estimate of $m_{K}^{(1)}(u,v,\sigma)$. 
\begin{lem}\label{lem4.3}
Let $\sigma_1$ be a large fixed positive real number, and suppose 
$\sigma\in(1/2,\sigma_1]$.
Then there exists a positive real number $J_K^{(1)}(\sigma_1)$ such that 
for all $\alpha,\beta\in\mathbb{R}$ with $|\alpha|+|\beta|\geq J_K^{(1)}(\sigma_1)$,
\[|m_{K}^{(1)}(u,v,\sigma)|\leq\exp(A_K(|\alpha|+|\beta|)^{\frac{3}{4}\frac{1}{\sigma}})\]
for all $(u,v)\in D(\alpha,1/2)\times D(\beta,1/2)$.
Here the constant $J_K^{(1)}(\sigma_1)$ depends only on $K$ and $\sigma_1$, and the constant $A_K>0$ depends only on $K$.
\end{lem}
\begin{proof}
We first see that
\[|m_{K,p}(u,v,\sigma)|
\leq\int_0^1\exp\{-\mathfrak{I}(ua_p^{(1)}(t)+va_p^{(2)}(t))\}dt.\]
By the conditions for $u$ and $v$, the absolute values of their imaginary parts are less than $1/2$. 
Hence we have  
\[|-\mathfrak{I}(ua_p^{(1)}(t)+va_p^{(2)}(t))|\ll\frac{d\log p}{p^\sigma}\]
with an absolute implied constant. 
Summing up over $p\leq P_0$, we obtain 
\[|m_{K}^{(1)}(u,v,\sigma)|\leq\exp\Big(Ad\sum_{p\leq P_0}\frac{\log p}{p^\sigma}\Big)
\leq\exp(A_1dP_0^{\frac{1}{2}}\log P_0).\]
Here $A$ and $A_1>0$ are absolute constants. 
With the choice of $P_0$, the desired estimate follows if $|\alpha|+|\beta|\geq J_K^{(1)}(\sigma_1)$. 
\end {proof}

\vspace{3mm}

The estimate for $m_{K}^{(2)}(u,v,\sigma)$ requires more work than $m_{K}^{(1)}(u,v,\sigma)$.
\begin{lem}\label{lem4.4}
Let $\sigma_1$ be a large fixed positive real number, and suppose $\sigma\in(1/2,\sigma_1]$.
Then there exists a positive real number $J_K^{(2)}(\sigma)$ such that for all $\alpha,\beta\in\mathbb{R}$ with $|\alpha|+|\beta|\geq J_K^{(2)}(\sigma)$, 
\[|m_{K}^{(2)}(u,v,\sigma)|
\leq\exp\Big(-\frac{c_K(\sigma_1)}{2\sigma-1}(|\alpha|+|\beta|)^{\frac{1}{\sigma}}
(\log(|\alpha|+|\beta|))^{\frac{1}{\sigma}-1}\Big)\]
for all $(u,v)\in D(\alpha,1/2)\times D(\beta,1/2)$.
The constant $J_K^{(2)}(\sigma)$ depends only on $K$ and $\sigma$, and the constant $c_K(\sigma_1)>0$ depends only on $K$ and $\sigma_1$.
\end{lem}
\begin{proof}
As before we express $m_{K,p}(u,v,\sigma)$ as $1-\mu+R$. 
We want to define $\log m_{K,p}(u,v,\sigma)$ again.
We let $\epsilon>0$ sufficiently small, and we assume $|\alpha|+|\beta|$ is greater than some positive absolute constants.
For $p\geq P_0$, we have
\begin{equation}\label{eq4.2}
\frac{(|\alpha|+|\beta|)d\log p}{p^{\sigma}}
\leq\frac{(|\alpha|+|\beta|)d\log P_0}{P_0^{\sigma}}, 
\end{equation}
hence \eqref{eq4.2} is less than $\epsilon$ if we choose $c_0=c_0(K)$ appropriately.
Then $|\mu|$ and $|R|$ are less than 1/4, and $\log m_{K,p}(u,v,\sigma)$ is defined and calculated as  
\[\log m_{K,p}(u,v,\sigma)=-\mu+R+O(|\mu|^2+|R|^2).\]
We first replace $\mu$ with the real number  
\[\frac{\alpha^2+\beta^2}{4}\sum_{m=1}^\infty
\frac{\Lambda_K(p^m)^2}{p^{2m\sigma}}.\]
The error of this replacement is estimated as
\begin{align*}
&\Big|\frac{u^2+v^2}{4}\sum_{m=1}^\infty
\frac{\Lambda_K(p^m)^2}{p^{2m\sigma}}
-\frac{\alpha^2+\beta^2}{4}\sum_{m=1}^\infty
\frac{\Lambda_K(p^m)^2}{p^{2m\sigma}}\Big|\\
&\hspace{15mm}\ll(|\alpha|+|\beta|)\sum_{m=1}^\infty\frac{\Lambda_K(p^m)^2}{p^{2m\sigma}} 
\end{align*}
with an absolute implied constant, since $(u,v)\in D(\alpha,1/2)\times D(\beta,1/2)$. 
For the estimate on $R$, we have
\[|R|\ll(|u|+|v|)^3\Big(\sum_{m=1}^\infty\frac{\Lambda_K(p^m)}{p^{m\sigma}}\Big)^3
\ll(|\alpha|+|\beta|)^3\frac{d^3\log^3p}{p^{3\sigma}}\]
for $p\geq P_0$, where all the implied constants are absolute. 
We next estimate $|\mu|^2$. 
Since $(|\alpha|+|\beta|)d\log P_0/P_0^\sigma$ is sufficiently small, we also have 
\begin{align*}
|\mu|^2&\ll(|u|+|v|)^4\frac{d^4\log^4p}{p^{4\sigma}}
\leq(|\alpha|+|\beta|)\frac{d\log P_0}{P_0^\sigma}(|\alpha|+|\beta|)^3\frac{d^3\log^3p}{p^{3\sigma}}\\
&\leq(|\alpha|+|\beta|)^3\frac{d^3\log^3p}{p^{3\sigma}}
\end{align*}
with an absolute implied constant. 
Furthermore, we clearly see that 
\[|R|^2\leq\frac{|R|}{4}\ll(|\alpha|+|\beta|)^3\frac{d^3\log^3p}{p^{3\sigma}}.\]
By the above estimates, we obtain 
\begin{align}\label{eq4.3}
&\left|\log m_{K,p}(u,v,\sigma)+\frac{\alpha^2+\beta^2}{4}
\sum_{m=1}^\infty\frac{\Lambda_K(p^m)^2}{p^{2m\sigma}}\right|\nonumber\\
&\leq\Big|\frac{\alpha^2+\beta^2}{4}\sum_{m=1}^\infty
\frac{\Lambda_K(p^m)^2}{p^{2m\sigma}}-\mu\Big|
+\Big|\log m_{K,p}(u,v,\sigma)+\mu\Big|\nonumber\\
&\leq A(|\alpha|+|\beta|)
\sum_{m=1}^\infty\frac{\Lambda_K(p^m)^2}{p^{2m\sigma}}
+A(|\alpha|+|\beta|)^3\frac{d^3\log^3p}{p^{3\sigma}}, 
\end{align}
where $A>0$ is an absolute constant.
We assume that $|\alpha|+|\beta|$ is greater than some absolute constants. 
Then the first term of \eqref{eq4.3} is less than 
\[\frac{(|\alpha|+|\beta|)^2}{16}\sum_{m=1}^\infty\frac{\Lambda_K(p^m)^2}{p^{2m\sigma}}.\]
As a result, we obtain  
\begin{align*}
\mathfrak{R}\log m_{K,p}(u,v,\sigma)
&\leq-\frac{\alpha^2+\beta^2}{4}\sum_{m=1}^\infty\frac{\Lambda_K(p^m)^2}{p^{2m\sigma}}
+\frac{(|\alpha|+|\beta|)^2}{16}\sum_{m=1}^\infty\frac{\Lambda_K(p^m)^2}{p^{2m\sigma}}\nonumber\\
&\hspace{4.5cm}+A(|\alpha|+|\beta|)^3\frac{d^3\log^3p}{p^{3\sigma}}\nonumber\\
&\leq-\frac{(|\alpha|+|\beta|)^2}{16}
\sum_{m=1}^\infty\frac{\Lambda_K(p^m)^2}{p^{2m\sigma}}
+A(|\alpha|+|\beta|)^3\frac{d^3\log^3p}{p^{3\sigma}}.
\end{align*}
Hence we have 
\begin{align}
|m_K^{(2)}(u,v,\sigma)|\label{eq4.4}
&=\exp\Big(\sum_{p\geq P_0}\mathfrak{R}\log m_{K,p}(u,v,\sigma)\Big)\\
&\leq\exp\Big(-\frac{(|\alpha|+|\beta|)^2}{16}
\sum_{p\geq P_0}\sum_{m=1}^\infty\frac{\Lambda_K(p^m)^2}{p^{2m\sigma}}\nonumber\\
&\hspace{3cm}+A(|\alpha|+|\beta|)^3\sum_{p\geq P_0}\frac{d^3\log^3p}{p^{3\sigma}}\Big)\nonumber. 
\end{align}
Then, we consider the lower bound on the following Dirichlet series
\[\sum_{p\geq P_0}\sum_{m=1}^\infty\frac{\Lambda_K(p^m)^2}{p^{2m\sigma}}.\]
For a rational prime number $p$, we write the prime ideal factorization of $p$ as 
\[p=\mathfrak{p}_1^{e(\mathfrak{p}_1,p)}
\cdots\mathfrak{p}_{g(p)}^{e(\mathfrak{p}_{g(p)},p)},\]
with $N(\mathfrak{p}_j)=p^{f(\mathfrak{p}_j,p)}$ for $j=1, \ldots, g(p)$. 
Suppose that $p$ is any rational prime that completely splits in $K$. 
Then the coefficient $\Lambda_K(p^m)$ is equal to $d\log p$ for any $m$ by the formula  \eqref{eq2.3}.
Thus we have 
\[\sum_{p\geq P_0}\sum_{m=1}^\infty\frac{\Lambda_K(p^m)^2}{p^{2m\sigma}}
\geq\sideset{}{'}\sum_{p\geq P_0}\sum_{m=1}^\infty
\frac{d^2\log^2p}{p^{2m\sigma}}
\geq\sideset{}{'}\sum_{p\geq P_0}\frac{d^2\log^2p}{p^{2\sigma}},\]
where $\sum'$ means the summation restricted to primes that completely split in $K$. 
The Chebotarev density theorem gives 
\[\frac{\#\{p\leq x~|~\text{$p$ is a prime number that completely splits in $K$}\}}
{\#\{p\leq x~|~\text{$p$ is a prime number}\}}
\to\frac{1}{[L:\mathbb{Q}]},\]
as $x\to\infty$, where $L$ denotes the Galois closure of $K/\mathbb{Q}$. 
Then we obtain  
\begin{equation}\label{eq4.5}
\sideset{}{'}\sum_{p\geq M}\frac{\log^2 p}{p^{2\sigma}}
\geq\frac{1}{[L:\mathbb{Q}]}\frac{1}{2(2\sigma-1)}M^{1-2\sigma}\log M
\end{equation}
for $M\geq M_0(\sigma_1)$ by the Prime Number Theorem, where the constant $M_0(\sigma_1)>0$ depends only on $\sigma_1$.
Next we consider the term 
\[A(|\alpha|+|\beta|)^3\sum_{p\geq P_0}\frac{d^3\log^3p}{p^{3\sigma}}\]
in \eqref{eq4.4}.
At first, we calculate this term as 
\[\leq Ad^2\frac{(|\alpha|+|\beta|)d\log P_0}{P_0^\sigma}
(|\alpha|+|\beta|)^2\sum_{p\geq P_0}\frac{\log^2p}{p^{2\sigma}}.\]
We replace the constant $c_0=c_0(K)$ smaller so that 
\[Ad^2\frac{(|\alpha|+|\beta|)d\log P_0}{P_0^\sigma}\leq\frac{1}{128}\frac{1}{[L:\mathbb{Q}]}.\]
By the Prime Number Theorem, we also obtain
\begin{equation}\label{eq4.6}
\sum_{p\geq M}\frac{\log^2p}{p^{2\sigma}}
\leq\frac{2}{2\sigma-1}M^{1-2\sigma}\log M
\end{equation}
for $M\geq M'_0(\sigma)$, where the constant $M'_0(\sigma)>0$ depends only on $\sigma$. 

We assume $|\alpha|+|\beta|\geq J_K^{(2)}(\sigma)$ for some positive constant $K^{(2)}(\sigma)$. 
By the estimates \eqref{eq4.4}, \eqref{eq4.5}, and \eqref{eq4.6}, applying the choice of $P_0$, we complete the proof of the present lemma. 
\end{proof}

\begin{proof}[Proof of Lemma \ref{lem4.2}]
This is a direct consequence of Lemmas \ref{lem4.3} and \ref{lem4.4}.
\end{proof}

\section{Proof of the Theorem \ref{thm2.3}}\label{sec5}
By Proposition \ref{prop2.2}, the function $m_K(u,v,\sigma)$ is integrable over $\mathbb{C}$ as a function of $w=u+iv$, then we define 
\[M_{K,\sigma}(z)=\int_{\mathbb{C}}m_K(u,v,\sigma)\psi_{-z}(w)|dw| .\]
Moreover, the function $m_K(u,v,\sigma)$ belongs to the class $\Lambda$ introduced in $\S\S$\ref{sec2.2}, hence we have 
\[\tilde{M}_{K,\sigma}(w)=\int_\mathbb{C}M_{K,\sigma}(z)\psi_{w}(z)|dz|=m_K(u,v,\sigma)\]
for $w=u+iv$.  
By the obvious equation $\tilde{M}_{K,\sigma}(-w)=\overline{\tilde{M}_{K,\sigma}(w)}$ from the definition, we find that $M_{K,\sigma}(z)$ is real valued. 
In fact, we have 
\begin{align*}
\overline{M_{K,\sigma}(z)}
&=\int_{\mathbb{C}}\overline{\tilde{M}_{K,\sigma}(w)}\psi_{z}(w)|dw|
=\int_{\mathbb{C}}\tilde{M}_{K,\sigma}(-w)\psi_{z}(w)|dw|\nonumber\\
&=\int_{\mathbb{C}}\tilde{M}_{K,\sigma}(w)\psi_{z}(-w)|dw|
=M_{K,\sigma}(z).
\end{align*}
Then we prove Theorem \ref{thm2.3}. 
By the assumption on the test function $\Phi$, we have the inversion formula
\[\Phi(z)=\int_{\mathbb{C}}\tilde{\Phi}(w)\psi_{-w}(z)|dw|.\]
Hence the integral 
\[\frac{1}{T}\int_0^T \Phi\Big(\frac{\zeta'_K}{\zeta_K}(\sigma+it)\Big)dt\]
is equal to 
\begin{equation}\label{eq5.1}
\int_{|w|\leq A}\tilde{\Phi}(w)\frac{1}{T}\int_0^T
\psi_{-w}\Big(\frac{\zeta'_K}{\zeta_K}(\sigma+it)\Big)dt|dw|
+O\Big(\int_{|w|>A}|\tilde{\Phi}(w)|~|dw|\Big), 
\end{equation}
where $A=(\log T)^\delta$. 
By Proposition \ref{prop2.1}, the first term of \eqref{eq5.1} is equal to 
\[\int_{|w|\leq A}\tilde{\Phi}(w)\tilde{M}_{K,\sigma}(-w)|dw|
+O_{K,\sigma_1}\Big(\exp(-c_K(\log T)^{\frac{2}{3}\theta})
\int_{|w|\leq A}|\Phi(w)|~|dw|\Big). \]
By the fact $|\tilde{M}_{K,\sigma}(-w)|\leq1$ coming from the definition, the error when we replace the integral over $|w|\leq(\log T)^\delta$ with the integral over $\mathbb{C}$ is
\[\ll\int_{|w|>A}|\tilde{\Phi}(w)|~|dw|.\]
Hence the remainder work is investigating the integral 
\begin{equation}\label{eq5.2}
\int_\mathbb{C}\tilde{\Phi}(w)\tilde{M}_{K,\sigma}(-w)|dw|. 
\end{equation}
We replace $\tilde{M}_{K,\sigma}(-w)$ with $\overline{\tilde{M}_{K,\sigma}(w)}$ and apply Parseval's formula, then the integral \eqref{eq5.2} is equal to 
\[\int_{\mathbb{C}}\Phi(z)\overline{M_{K,\sigma}(z)}|dz|=\int_{\mathbb{C}}\Phi(z)M_{K,\sigma}(z)|dz|.\]
Therefore, we complicate the proof of Theorem \ref{thm2.3}. \qed

\section{Further results on $M$-functions by Guo's method}\label{sec6}
According to Theorem \ref{thm2.3}, the function $M_{K,\sigma}(z)$ is regarded as the density function corresponding to the value-distribution of $(\zeta'_K/\zeta_K)(\sigma+it)$. 
Such a density function is sometimes called ``$M$-function'' since the paper of Ihara \cite{Ihara}. 
As we have seen, Theorem \ref{thm2.3} is deduced from Proposition \ref{prop2.1} and \ref{prop2.2}. 
Hence, if we have the analogue of these propositions for a given zeta or $L$-function, then it is possible to prove the existence of the $M$-function on such a zeta or $L$-function. 

For $L$-functions with some nice properties, the result corresponding to Proposition \ref{prop2.1} holds if we suppose a result analogous to Lemma \ref{lem3.4}. 
Hence the main difficulty should exist in the poof of Proposition \ref{prop2.2}. 
See the proof of Lemma \ref{lem4.4}. 
There, we investigated the lower bound of 
\begin{equation}\label{eq6.1}
\sum_{p\geq M}\sum_{m=1}^\infty\frac{\Lambda_K(p^m)^2}{p^{2m\sigma}}.
\end{equation}
In general, we should consider 
\begin{equation}\label{eq6.2}
\sum_{p\geq M}\sum_{m=1}^\infty\frac{|\Lambda_\ast(p^m)|^2}{p^{2m\sigma}},
\end{equation}
where $\Lambda_\ast(n)$ is the coefficient of $n^{-s}$ in the Dirichlet series defining the logarithmic derivative of a zeta or $L$-function. 
The reason why $\Lambda_K(p^m)$ is replaced by $|\Lambda_\ast(p^m)|$ comes from the alternation of the functions $a_p^{(1)}(t)$ and $a_p^{(2)}(t)$ in \eqref{eq3.25} and \eqref{eq3.26}, which should become more complicated unless $\Lambda_\ast(p^m)\in\mathbb{R}$:  
\[a_p^{(1)}(t)=-\sum_{m=1}^\infty\frac{\mathfrak{R}\Lambda_\ast(p^m)}{p^{m\sigma}}\cos(2\pi mt)+\sum_{m=1}^\infty\frac{\mathfrak{I}\Lambda_\ast(p^m)}{p^{m\sigma}}\sin(2\pi mt)\]
and
\[a_p^{(2)}(t)=-\sum_{m=1}^\infty\frac{\mathfrak{R}\Lambda_\ast(p^m)}{p^{m\sigma}}\sin(2\pi mt)-\sum_{m=1}^\infty\frac{\mathfrak{I}\Lambda_\ast(p^m)}{p^{m\sigma}}\cos(2\pi mt).\]

The general method seeking for the lower bound \eqref{eq6.2} is unlikely to exist. 
There is just an individual case in which we can obtain an available lower bound, and the easiest case is the Dirichlet $L$-functions on $\mathbb{Q}$. 
Let $\chi$ be a primitive Dirichlet character of conductor $q$. 
In this case, we have $\Lambda_\ast(p^m)=\Lambda_\chi(p^m)=\chi(p^m)\Lambda(p^m)$, then $|\Lambda_\chi(p^m)|=\log p$ for almost all $p$. 
Then we have the same lower bound as \eqref{eq6.1} in the case $K=\mathbb{Q}$.

Another case in which a similar argument remains possible is $L$-functions attached to cusp forms. 
Let $f$ be a primitive cusp form of weight $k$ with respect to the modular group $SL_2(\mathbb{Z})$ with the Fourier expansion
\[f(z)=\sum_{n=1}^\infty \lambda_f(n)n^{(k-1)/2}e^{2\pi inz}. \]
It is notable that an analogue of Lemma \ref{lem3.4} has been shown by Luo \cite{Luo}, that is 
\[N_f(\sigma,T)\ll T^{1-\frac{1}{72}(\sigma-\frac{1}{2})}\log T\]
uniformly for $1/2\leq\sigma\leq1$, where $N_f(\sigma,T)$ denotes the number of zeros $\rho=\beta+i\gamma$ of the $L$-function $L(s,f)$ with $\beta\geq\sigma$ and $0\leq\gamma\leq T$. 
In this case, we have $\Lambda_\ast(p^m)=\Lambda_f(p^m)=(\alpha_f(p)^m+\beta_f(p)^m)\log p$ for almost all $p$, where the parameters $\alpha_f(p)$ and $\beta_f(p)$ satisfy $\alpha_f(p)+\beta_f(p)=\lambda_f(p)$ and $\alpha_f(p)\beta_f(p)=1$. 
Hence we have 
\[\sum_{p\geq M}\sum_{m=1}^\infty\frac{\Lambda_f(p^m)^2}{p^{2m\sigma}}
\geq\sideset{}{'}\sum_{p\geq M}\sum_{m=1}^\infty
\frac{(\alpha_f(p)^m+\beta_f(p)^m)^2\log^2p}{p^{2m\sigma}}
\geq\sideset{}{'}\sum_{p\geq M}\frac{\log^2p}{p^{2\sigma}},\]
where $\sum'$ means the summation restricted to primes that satisfy 
$\lambda_f(p)\geq1$. 
Next we investigate the distribution of values of $\lambda_f(p)$. 
Murty \cite[Theorem 4.]{Murty} proved that for any $\epsilon>0$, there exists a set of prime numbers of positive density, such that $\lambda_f(p)>\sqrt{2}-\epsilon$ holds. 
Therefore, we obtain the lower bound similar to \eqref{eq4.5} such that 
\[\sideset{}{'}\sum_{p\geq M}\frac{\log^2p}{p^{2\sigma}}
\geq\frac{c}{2\sigma-1}M^{1-2\sigma}\log M.\] 
for $M\geq M_0(\sigma_1)$ with some positive constant $c$. 

In these two cases, the analogue of Theorem \ref{thm2.3} does holds. 
Let $\ast$ mean a primitive Dirichlet character $\chi$ or a primitive cusp form $f$. 
We define  
\[M_{\ast,\sigma}(z)=\int_{\mathbb{C}}m_\ast(u,v,\sigma)\psi_{-z}(w)|dw| \]
for $\sigma>1/2$, where
\[m_\ast(u,v,\sigma)=\prod_p\int_0^1\psi_w\Big(-\sum_{m=1}^\infty
\frac{\Lambda_\ast(p^m)}{p^{m\sigma}}e^{2\pi imt}\Big)dt .\]
Then the function $M_{\ast,\sigma}(z)$ satisfies the following statement: 

\begin{thm}\label{thm6.1}
Let $\sigma_1$ be a large fixed positive real number.
Let $\theta$ and $\delta$ be two positive real numbers with $\delta+\theta<1/2$.
Then there exists a positive real number 
$T_0=T_0(\theta,\delta,\sigma_1,\ast)$ such that for all $T\geq T_0$ 
and for all $\sigma$ in the interval $[1/2+(\log T)^{-\theta},\sigma_1]$, 
we have 
\[\frac{1}{T}\int_0^T \Phi\Big(\frac{L'}{L}(\sigma+it,\ast)\Big)dt
=\int_{\mathbb{C}}\Phi(z)M_{\ast,\sigma}(z)|dz|+E\]
for any $\Phi\in\Lambda$, 
where 
\[E\ll_{\ast,\sigma_1}\exp\big(-\frac{1}{4}(\log T)^{\frac{2}{3}\theta}\big)
\int_{|w|\leq(\log T)^\delta}|\tilde{\Phi}(w)|~|dw|+\int_{|w|>(\log T)^\delta}|
\tilde{\Phi}(w)|~|dw|.\]
The implied constant 
depends only on $\ast$ and $\sigma_1$. 
\end{thm}
The proof is completely the same as Theorem \ref{thm2.3}.

\end{document}